\newif\ifspringer
\springerfalse

\newif\ifelsevier
\elsevierfalse

\ifspringer
\documentclass[envcountsect,envcountsame]{svjour3}
\smartqed
\else
\ifelsevier
\documentclass[a4paper,fleqn]{cas-dc}
\else
\documentclass[a4paper]{article}
\fi
\fi

\usepackage{amsmath}
\usepackage{amsfonts}
\usepackage{amssymb}         
\usepackage{amsopn}
\usepackage{theorem}          
\usepackage{latexsym}
\usepackage{graphicx}        
\usepackage{mathrsfs}		
\usepackage{psfrag}
\usepackage{algorithm}
\usepackage{algorithmic}
\usepackage{color}
\usepackage{verbatim}
\usepackage{url}
\usepackage{enumerate}
\ifelsevier
\else
\usepackage{titlesec}
\fi
\usepackage{tikz}
\usepackage{hyperref}
\usetikzlibrary{positioning,chains,fit,shapes,calc,arrows}

\ifspringer
\else
\ifelsevier
\else
\fi
\fi

\ifspringer
\else
\ifelsevier
\else
\newtheorem{result}{\ }[section]
\theoremstyle{changebreak}                
\newtheorem{theorem}[result]{Theorem}

\newtheorem{lemma}[result]{Lemma}
\newtheorem{corollary}[result]{Corollary}
\newtheorem{proposition}[result]{Proposition}

\newenvironment{proof}
 {{\sl Proof.}\hspace*{1 ex}}%
 {{\nopagebreak\hspace*{\fill}$\Box$\par\vspace{12pt}}}
\newcommand{\qed}{\hfill \ensuremath{\Box}}
\fi
\fi

\newcommand{\transpose}[1]{{#1}^{\top}}
\newcommand{\cone}[1]{\mathsf{cone}(#1)}
\newcommand{\cvx}[1]{\mathsf{cvx}(#1)}
\newcommand{\COR}[1]{\mathsf{COR}(#1)}
\newcommand{\nCOR}[1]{\mathsf{\overline{COR}}(#1)}
\newcommand{\CUT}[1]{\mathsf{CUT}(#1)}
\newcommand{\nCUT}[1]{\mathsf{\overline{CUT}}(#1)}
\newcommand{\CONX}[1]{\mathcal{C}(#1)}

\newcommand{\diag}[1]{\mathsf{diag}(#1)}

\newcommand{\leo}[1]{{\color{black}{#1}}}

\def\K{{\cal K}}
\def\C{{\cal C}}

\begin{document}

\ifspringer

\title{Hardness of some optimization problems over correlation polyhedra}
\author{Alberto Caprara\and Fabio Furini\and Claudio Gentile\and Leo Liberti\and Andrea Lodi}
\institute{A.~Caprara\at Universit\`a di Bologna, Italy ($\dag$ 2012) \and F.~Furini\at Department of Computer, Control and Management Engineering ``A.~Ruberti'' --- Sapienza University of Rome. Via Ariosto 25, 00185 Roma, Italy \\ \email{\url{fabio.furini@uniroma1.it}} \and C.~Gentile\at Istituto di Analisi dei Sistemi ed Informatica ``A. Ruberti'', Consiglio Nazionale delle Ricerche, Via dei Taurini 19, 00185 Roma, Italy \\ \email{\url{gentile@iasi.cnr.it}}\and L.~Liberti (\textit{corresponding author})\at LIX CNRS Ecole Polytechnique, Institut Polytechnique de Paris, 91128 Palaiseau, France \\ \email{\url{leo.liberti@polytechnique.edu}}\and A.~Lodi\at CornellTech, New York, USA; and IASI, Rome, Italy \\ \email{\url{al748@cornell.edu}}}
\maketitle
  
\else

\ifelsevier

\let\WriteBookmarks\relax
\def\floatpagepagefraction{1}
\def\textpagefraction{.001}
\shorttitle{Hardness over correlation polyhedra}    
\shortauthors{Caprara et al.}  
\title [mode = title]{Hardness of some optimization problems over correlation polyhedra}  

\tnotemark[1] 

\tnotetext[1]{} 

\author[1]{Alberto Caprara}
\fnmark[1]
\fntext[1]{$\dag$ 2012}
\affiliation[1]{organization={Universit\`a di Bologna},country={Italy}}

\author[2]{Fabio Furini}
\ead{furini@diag.uniroma1.it}
\affiliation[2]{organization={Department of Computer, Control and Management Engineering ``A.~Ruberti'' --- Sapienza University of Rome},addressline={Via Ariosto 25},postcode={00185},city={Roma},country={Italy}}

\author[3]{Claudio Gentile}
\ead{gentile@iasi.cnr.it}
\affiliation[3]{organization={Istituto di Analisi dei Sistemi ed Informatica ``A.~Ruberti'', Consiglio Nazionale delle Ricerche},addressline={Via dei Taurini 19},postcode={00185},city={Roma},country={Italy}}

\author[4]{Leo Liberti}[orcid=0000-0003-3139-6821]
\cormark[4]
\fnmark[4]
\fntext[4]{LL gratefully acknowledges grants from CNR and INDAM.}
\ead{leo.liberti@polytechnique.edu}
\ead[url]{www.lix.polytechnique.fr/~liberti}
\affiliation[4]{organization={LIX CNRS, Ecole Polytechnique, Institut Polytechnique de Paris},postcode={91128},city={Palaiseau},country={France}}

\author[5]{Andrea Lodi}
\ead{al748@cornell.edu}
\affiliation[5]{organization={CornellTech},city={New York},country={USA}}
\else

\thispagestyle{empty}
\begin{center} 

{\LARGE Hardness of some optimization problems over correlation polyhedra}
\par \bigskip
{\sc Alberto Caprara${}^1$, Fabio Furini${}^2$, Claudio Gentile${}^3$, Leo Liberti${}^{4}$, Andrea Lodi${}^{5,3}$} 
\par \bigskip
\begin{minipage}{15cm}
\begin{flushleft}
{\small
\begin{itemize}
\item[${}^1$] {\it Universit\`a di Bologna, Italy} ($\dag$ 2012)
\item[${}^2$] {\it Department of Computer, Control and Management Engineering ``A.~Ruberti'' --- Sapienza University of Rome. Via Ariosto 25, 00185 Roma, Italy} \\ Email:\url{fabio.furini@uniroma1.it}
\item[${}^3$] {\it Istituto di Analisi dei Sistemi ed Informatica ``A.~Ruberti'', Consiglio Nazionale delle Ricerche, Via dei Taurini 19, 00185 Roma, Italy} \\ Email:\url{gentile@iasi.cnr.it}
\item[${}^4$] {\it LIX CNRS, \'Ecole Polytechnique, Institut Polytechnique de Paris, F-91128 Palaiseau, France} \\ Email:\url{liberti@lix.polytechnique.fr}
\item[${}^5$] {\it CornellTech, New York, USA} \\ Email:\url{al748@cornell.edu}
\end{itemize}
}
\end{flushleft}
\end{minipage}
\par \medskip \today
\end{center}
\par \bigskip
\fi
\fi

\begin{abstract}
  We prove the \textbf{NP}-hardness, using Karp reductions, of some problems related to the correlation polytope and its corresponding cone, spanned by all of the $n\times n$ rank-one matrices over $\{0,1\}$. The problems are: membership, rank of the decomposition, and a ``relaxed rank'' obtained from relaxing the zero-norm expression for the rank to an $\ell_1$ norm. While membership and rank are natural problems for any matrix cone, the relaxed rank problem occurs in some signal processing and statistical applications. 
  
\ifspringer
  \else
  \ifelsevier
  \else
 \fi
 \fi
\end{abstract}

\ifelsevier
\maketitle
\fi

\section{Introduction}
\label{s:intro}
For any finite set $\mathcal{M}$ of matrices one can define the cone $\cone{\mathcal{M}}$ and the polytope $\cvx{\mathcal{M}}$. The fundamental question in this setting is that of membership: given a matrix $M'$, does it belong to the cone or the convex hull of $\mathcal{M}$? We note that $M'\in\cone{\mathcal{M}}$ is equivalent to
\begin{equation}
  \exists p\in\mathbb{R}^{|\mathcal{M}|}_+\quad M' = \sum_{M\in\mathcal{M}} p_M M,\label{eq:gencone}
\end{equation}
while $M'\in\cvx{\mathcal{M}}$ constrains $p$ to be in the unit simplex, i.e.~it adds the constraint $\sum_{\mu\in\mathcal{M}} p_\mu=1$ to Eq.~\eqref{eq:gencone}.
A second question is that of rank: what is the minimum number of elements of $\mathcal{M}$ using which one can represent $M$? Or, equivalently, what is the minimum possible value of $\|p\|_0$ for $M$? A third relevant question is that of minimizing the relaxed rank $\|p\|_1$, given that the ``mathematics of sparsity'' \cite{candes} promises to turn a problem involving binary variables into one involving continuous variables only. Obviously, the relaxed rank question only makes sense for $\cone{\mathcal{M}}$ since the convex hull fixes $\|p\|_1=1$ (recall that $p\ge 0$). 

We consider the \textbf{NP}-hardness of these three questions for a special case where $\mathcal{M}$ is the set of rank-one matrices with entries in $\{0,1\}$. The rank and relaxed rank minimization problems then become decision questions: is the rank, or relaxed rank, less than a given threshold or not? For rank and relaxed rank decision problems we prove \textbf{NP}-completeness ``with promise'', i.e.~we assume that $M$ satisfies Eq.~\eqref{eq:gencone} or its convex hull counterpart, otherwise hardness follows trivially from membership. The fulfilment of the promise is given by a single bit of information: i.e.~we know that $M$ can be decomposed as per Eq.~\eqref{eq:gencone}, but we do not have the weights $p_\mu$ in the decomposition. Therefore, the ``rank of $p$'' and ``relaxed rank of $p$'' $\le$ threshold questions are non-trivial. We provide Karp reductions for every one of our hardness proofs: such reductions are stronger than Cook-Turing reductions (typically afforded by the optimization/separation/membership equivalence \cite{gls}) because they are based on a single query to the target problem. 

We now consider $\mathcal{M}$ as a set of rank-1 symmetric matrices with entries in sets of two integers, either $\{0,1\}$ or $\{-1,1\}$. We focus mostly on the former, but obtain results on both. Given a symmetric matrix $\Gamma$ of dimension $n\times n$, with entries $\Gamma_{ij}=\Gamma_{ji}\in\mathbb{R}$, we consider its decomposition in the following form:
\begin{equation}
  \Gamma = \sum\limits_{k\ge 0} p_k x_k\transpose{x}_k,
  \label{eq:decomp}
\end{equation}
where $x_k$ is the vector in $\{0,1\}^n$ providing the base-2 representation of the integer $k$, and the multiplier vector $p$ is non-negative. From this, it follows that $0\le k\le P_n=2^n-1$, so $p\in\mathbb{R}^{2^n}_+$, and that all $n\times n$ matrices $X^k=x_k\transpose{x}_k$ are rank-1 symmetric and defined on $\{0,1\}$. Therefore, no matrix $\Gamma$ can be decomposed in the form Eq.~\eqref{eq:decomp} unless it is symmetric and all its entries are non-negative. For a given $\Gamma$, we are also interested in the smallest number of terms of the decomposition having $p_k>0$, i.e.~its rank. If we let $\CONX{n}$ be the cone spanned by all of the rank-one boolean symmetric matrices $X^k$, then the feasibility of Eq.~\eqref{eq:decomp} is the membership problem for the conic hull $\CONX{n} \triangleq \cone{X}$, where $X=\{X^k\;|\;0\le k\le P_n\}$.

Whenever Eq.~\eqref{eq:decomp} is subject to the constraint $\sum_{k\ge 0} p_k=1$, we obtain the definition of the \textit{correlation polytope} of order $n$, denoted by $\COR{n}$. This also establishes $\CONX{n}$ as the \textit{correlation cone} of order $n$. 

The rank problem for a given $\Gamma\in\CONX{n}$ can be formulated as the minimization of the zero norm of $p$ subject to Eq.~\ref{eq:decomp}, i.e.
\begin{equation}
  \min\{\|p\|_0 \;|\; \Gamma = \sum\limits_{0\le k\le P_n} p_k X^k \land p\ge 0\},
  \label{eq:rank}
\end{equation}
which an obvious reformulation turns into an exponentially-sized Mixed-Integer Linear Program (MILP). We observe that the linear system Eq.~\eqref{eq:decomp} has $N=n(n+1)/2$ equations, which implies that at most $N$ of the exponentially many variables $p_k$ need be nonzero: this provides an upper bound to the rank. Following the well-known relaxation of the zero-norm to $\ell_1$ norm \cite{candestao2005}, we obtain an exponentially-sized Linear Programming (LP) formulation that computes the \textit{relaxed rank} of a matrix $\Gamma\in\CONX{n}$, i.e.
\begin{equation}
  \min\{\|p\|_1\;|\; \Gamma = \sum\limits_{0\le k\le P_n} p_k X^k \land p\ge 0\}.
  \label{eq:relrank}
\end{equation}
Note that, since $p\ge 0$, we have $\|p\|_1=\sum\limits_{0\le k\le P_n} p_k=\langle\mathbf{1},p\rangle$.

Membership and rank problems are fundamental problems for any matrix decomposition. The relaxed rank becomes meaningful whenever the results from the ``mathematics of sparsity'' \cite{candes,moitra} apply. These results aim at finding sparse vectors as solutions of various problems. These results are relevant when the relaxed rank problem is used as a proxy to the rank problem: namely, this can be done when the number of constraints of a (large) LP is at least of the order of magnitude of the logarithm of the number of variables.  Such an approximation is justified in our case because $N=O(n^2)$ and $\log(P_n)=O(n)$. This makes the relaxed rank problem likely to be easier to solve, in practice, than the rank problem, although both are equally hard in theory. We note that the relaxed rank problem appears in \cite{furini}.

In summary, we consider the following decision problems:
\begin{enumerate}
\item conic hull membership: is $\Gamma\in\CONX{n}$? \label{prob1}
\item conic hull rank: if $\Gamma\in\CONX{n}$, is $\|p\|_0\le\rho$?
\item conic hull relaxed rank: if $\Gamma\in\CONX{n}$, is $\|p\|_1\le\rho$? \label{prob3}
\item convex hull membership: $\Gamma\in\COR{n}$? \label{prob4}
\item convex hull rank: if $\Gamma\in\COR{n}$, is $\|p\|_0\le\rho$? \label{prob5}
\end{enumerate}
We recall that the relaxed rank of the convex hull is not considered because $\sum_k p_k=\|p\|_1=1$ by definition, so every instance in $\COR{n}$ has answer YES as long as $\rho\ge 1$, and every other instance has answer NO: the problem therefore reduces to checking whether $\rho\ge 1$. We state outright that Problem \ref{prob4} in the above list is already known to be \textbf{NP}-hard by \cite{pitowsky}. At first look, Problem \ref{prob5} above would appear to be related to \cite[Problem 5]{bereg}; but, on closer inspection, the vertex set is part of the input in \cite{bereg}, which rules out our problems since Eq.~\eqref{eq:decomp} lists exponentially many terms in $n$. 

In this paper we prove four main theorems, to the effect that membership, rank, and relaxed rank problems with respect to $\CONX{n}$, as well as rank with respect to $\COR{n}$, are all \textbf{NP}-hard. Together with the result in \cite{pitowsky}, this paper settles all the problems in the above list. As corollaries, we also prove \textbf{NP}-hardness of: membership and rank problems for the cut polytope \cite{dezaLaurent}; membership and rank problems for cut and correlation polytopes without the zero element; and membership, rank, and relaxed rank problem for the cut cone. We note that the boolean quadric polytope (BQP) \cite{bqp} is essentially the same as $\COR{n}$, so it inherits the same hardness results (see Sect.~\ref{s:bqp}).

The rest of this paper is organized as follows. A literature review is given in Sect.~\ref{s:litrev}, with the purpose of justifying our interest in the above problems, and clarifying the reasons why we need new hardness proofs for the problems in the list above. In Sect.~\ref{s:prelim}, we give a few easy necessary conditions for matrices $\Gamma$ to be in $\CONX{n}$. The four main results and their corollaries are given in Sect.~\ref{s:hardness}. Some polynomial cases and other related results are discussed in Sect.~\ref{s:related}. Sect.~\ref{s:concl} concludes the paper.

\section{Literature review}
\label{s:litrev}
This section gives an overview of the literature about complexity in the correlation polytopes and some other related polytopes and polyhedra. Our study of the relevant literature shows that the \textbf{NP}-hardness of Problems \ref{prob1}-\ref{prob5} of the list in Sect.~\ref{s:intro} is only clearly established, with Karp reductions, for Problem \ref{prob4}.

In the following, we shall cite \cite{gls} for two purposes: deriving \textbf{NP}-hardness proofs concerning two affinely isomorphic polytopes, and to support the polynomial-time equivalence, under Cook-Turing reductions, of linear optimization separation and membership for rational polyhedra. 

\subsection{The correlation polytope}
The correlation polytope \cite{pitowsky} is $\COR{n}\triangleq\cvx{X}$, i.e.~the convex hull of $X$. Its formulation is
\begin{equation}
  \left.\begin{array}{rcl}
    \sum\limits_{0\le k\le P_n} p_k x_k\transpose{x}_k &=& \Gamma \\
    \sum\limits_{0\le k\le P_n} p_k &=& 1.
    \end{array}\right\}
  \label{eq:cor}
\end{equation}
In the space of the $p$ variables, $\COR{n}$ is the intersection of the cone $\CONX{n}$ (right-hand-side of Eq.~\eqref{eq:decomp}) with the affine subspace $\langle\mathbf{1},p\rangle=1$. In Eq.~\eqref{eq:decomp}, we note that the zero matrix bears zero contribution to the sum, independently of the value of $p_0$: instead of summing over $k\ge 0$, we could therefore sum over $k\ge 1$. This corresponds to removing the zero matrix from $\COR{n}$: we define the \textit{normalized correlation polytope} as $\nCOR{n}\triangleq\cvx{X\smallsetminus 0}$.

\subsection{The boolean quadric polytope}
\label{s:bqp}
The BQP of order $n$ is the convex hull of pairs $(x,y)$ for all $x\in\{0,1\}^n$ and $y\in\{0,1\}^{n(n-1)/2}$ that satisfy
\begin{eqnarray}
  \forall 1\le i\le j\le n\quad y_{ij} &\ge& 0 \\
  \forall 1\le i\le j\le n\quad y_{ij} &\le& x_i \\
  \forall 1\le i\le j\le n\quad y_{ij} &\le& x_j \\
  \forall 1\le i\le j\le n\quad y_{ij} &\ge& x_i + x_j - 1.
\end{eqnarray}
For any $x\in\{0,1\}$, it is well-known \cite{fortet} that $y_{ij}=x_ix_j$ for all $i<j\le n$. Note that, should we extend $y$ to also include $y_{ii}$, we would obtain $y_{ii}=x_ix_i=x_i^2=x_i$ since $x_i\in\{0,1\}$. Thus, the BQP can also be written as the convex hull of all matrices $Y$ of the form
\[
Y = \left(\begin{array}{cccc}
  x_1 & y_{12} & \cdots & y_{1n} \\
  y_{21} & x_2 & \cdots & y_{2n} \\
  \vdots & \vdots & \ddots & \vdots \\
  y_{n1} & y_{n2} & \cdots & x_n
  \end{array}\right) =
  \left(\begin{array}{cccc}
  y_{11} & y_{12} & \cdots & y_{1n} \\
  y_{21} & y_{22} & \cdots & y_{2n} \\
  \vdots & \vdots & \ddots & \vdots \\
  y_{n1} & y_{n2} & \cdots & y_{nn}
  \end{array}\right),
\]
the first being simply the trivial symmetrization of the original $(x,y)$ vector, and the second being due to $y_{ii}=x_i$ for all $i\le n$. This form shows that $Y=x\transpose{x}$. Then, the BQP can be written as the convex hull of all $Y$ matrices derived as $x\transpose{x}$ where $x$ ranges over $\{0,1\}^n$, which is exactly the definition of $\COR{n}$. We note that the above mapping is a polynomial-time computable rational linear isomorphism. By \cite{gls}, this gives a valid polynomial-time reduction between corresponding decision problems. Therefore, the \textbf{NP}-hardness on $\COR{n}$ problems (with or without zero element) transfers to the boolean quadratic polytope (with or without zero element).

\subsection{The cut polytope and its cone}
\label{s:cutpolycone}
The correlation polytope is also related to the cut polytope $\CUT{n}$ \cite{dezaLaurent}, defined similarly to Eq.~\eqref{eq:cor}, but with all matrices $Y^k=y_k\transpose{y}_k$ where $y_k\in\{-1,1\}^n$ for all  $0\le k\le P_n$. Note that each vector $y_k\in\{-1,1\}^n$ can be written as $2x_k-\mathbf{1}_n$: through this correspondence we can define $y_k$ as the vector whose corresponding $x_k$ encodes the integer $k$ in binary. There is, however, a twist: while swapping zeros and ones in $x_k\in\{0,1\}$ one obtains another vector $x_h\in\{0,1\}$ such that $X^k\not=X^h$, if one swaps $\pm 1$ in $y_k$ one obtains another vector $-y_k$ that yields the same matrix $Y^k$: for example, the all-zero and all-one matrices in $\COR{n}$ both map to the all-one matrix in $\CUT{n}$ under the transformation $2x_k-\mathbf{1}$. Hence, while $\CONX{n}$ and $\COR{n}$ have $2^n$ generators, $\CUT{n}$ has $2^{n-1}$ generators: half with respect to $\COR{n}$. Cardinality-wise, $\COR{n}$ therefore corresponds to $\CUT{n+1}$ \cite{desimone90}.

Analogously to the correlation cone, the cut cone is defined as $\cone{Y}$ where $Y=\{Y^k=y_k\transpose{y}_k\;|\;y_k\in\{-1,1\}\land 0\le k\le P_n\}$ is the set of all symmetric $n\times n$ rank-1 matrices with entries over $\{-1,1\}$. It is shown in \cite{karzanov} that the separation problem for the cut cone is \textbf{NP}-hard. 

\subsection{Linear optimization}
\label{s:linopt}
It is known that linear optimization over $\COR{n}$ and $\CUT{n}$ is \textbf{NP}-hard \cite{bqp,dezaLaurent}.

We first consider Quadratic Boolean Optimization (QuBO)
\[\max\{\transpose{x}Qx+\transpose{c}x\;|\;x\in\{0,1\}^n\},\]
where we linearize each quadratic product, introducing a variable matrix $X=x\transpose{x}$. This yields a linear optimization problem $\max \langle Q,X\rangle$ over the set of vertices $X=x\transpose{x}$, where $x\in\{0,1\}^n$ of the $\COR{n}$ polytope. Since linear optimization on a set of points has the same optima as linear optimization on the convex hull of the points, linear optimization on $\COR{n}$ is as hard as QuBO, which is known to be \textbf{NP}-hard. This reduction is detailed but implicit in \cite[Eq.~(1)-(8)]{bqp}, and succinct but explicit in \cite[Eq.~(5.1.4), p.~54]{dezaLaurent}.

Next, we consider the \textsc{max cut} formulation \cite{dezaLaurent}
\[\max\big\{\sum_{i<j} w_{ij}\frac{1-y_iy_j}{2} \;|\;y\in\{-1,1\}^n\big\},\]
where $w$ are a set of given weights on unordered pairs $\{i,j\}$. We re-write this formulation as a linear matrix optimization problem $\max\{\langle W,Y\rangle\;|\;Y=y\transpose{y}\land y\in\{-1,1\}^n\}$, where $W=(w_{ij})$ is an $n\times n$ symmetric matrix. Since the feasible set $\{y\transpose{y}\;|\;y\in\{-1,1\}^n\}$ is the vertex set of $\CUT{n}$, we can take the convex hull of this vertex set and obtain the exponentially-sized LP $\max\{\langle W,Y\rangle\;|\;Y\in\CUT{n}\}$. Again, since linear optimization on a set of points has the same optima as linear optimization on the convex hull of the points, linear optimization on $\CUT{n}$ is as hard as \textsc{max cut}, which is known to be \textbf{NP}-hard. This reduction is mentioned in \cite[Eq.~(4.1.4), p.~38]{dezaLaurent}. 

Even though these two linear optimization problems are \textbf{NP}-hard, we cannot conclude that all their possible subsets of instances form \textbf{NP}-hard problems. We therefore need to explicitly prove \textbf{NP}-hardness for rank problems with respect to $\COR{n}$ and $\CUT{n}$, as well as relaxed rank problems for the corresponding cones (see Sect.~\ref{s:hardness}).

\subsection{Membership}
\label{s:membership}
We already mentioned that testing membership in $\COR{n}$ is \textbf{NP}-hard by \cite{pitowsky} (Problem \ref{prob4} from the list in Sect.~\ref{s:intro}). We note that \cite{pitowsky} provides a Karp reduction. We mention in passing that a Cook-Turing reduction is readily established by Sect.~\ref{s:linopt} and the polynomial-time equivalence of linear optimization, separation, and membership over $\COR{n}$ \cite{gls}. 

In \cite{dezaLaurent}, the \textbf{NP}-hardness of the corresponding membership problem in $\CUT{n}$ is derived as follows: \textsc{max cut} proves \textbf{NP}-hardness of linear optimization over $\CUT{n}$ (Sect.~\ref{s:linopt}). By \cite{gls}, linear optimization separation and membership in rational polytopes are polynomial-time equivalent problems under Cook-Turing reductions: whence membership in $\CUT{n}$ is \textbf{NP}-hard. Similarly, the fact that separation in the cut cone is \textbf{NP}-hard \cite{karzanov} implies that membership in that cone is also \textbf{NP}-hard. We note that these are Cook-Turing reductions: we will provide Karp reductions in Sect.~\ref{s:hardness}.

\subsection{Rank}
It is known that solving sparse linear systems of equations (\textsc{sparse linear system}) is \textbf{NP}-hard \cite{natarajan}. More precisely, there is a reduction from \textsc{exact cover by 3-sets} (X3C) to the (decision version of the) problem of computing a vector $p$ satisfying $\|Ap-b\|_2\le\epsilon$ for given $A,b,\epsilon>0$ such that $\|p\|_0$ is minimum. Given the real numbers setting due to the presence of the $\ell_2$ norm, the reduction assumes the Real RAM computational model \cite[\S 1.4]{preparata}. The requirement of a real computational model and the absence of the constraint $p\ge 0$ makes it doubtful that \textsc{sparse linear system} is a valid candidate for a source problem in view to reduce to rank minimization for $\COR{n}$ or $\CONX{n}$. Indeed, we use reductions from other problems in Sect.~\ref{s:hardness}. 

\section{Some necessary conditions for membership in $\CONX{n}$}
\label{s:prelim}
We start with some easy preliminary results in order to establish some necessary conditions for $\Gamma$ to belong to the cone $\CONX{n}$. A \textit{positive semidefinite} (PSD) matrix is a square symmetric matrix where all the eigenvalues are non-negative.
\begin{lemma}
  If $\Gamma\in\CONX{n}$, then it is PSD. \label{lem:necexact2}
\end{lemma}
\begin{proof}
  The $k$-th term of the sum decomposition Eq.~\eqref{eq:decomp} consists of a non-negative scalar $p_k$ which multiplies a rank-one matrix $X^k=x_k\transpose{x}_k$, where $x_k\in\{0,1\}^n$, for all $k\le P_n$. For any $y\in\mathbb{R}^n$, we have $\transpose{y}X^ky=\transpose{y}x_k\transpose{x}_ky=\transpose{(\transpose{x}_ky)}(\transpose{x}_ky)=\|\transpose{x}_ky\|_2^2\ge 0$. Hence, $X^k$ is a PSD matrix for all $k\le P_n$. The non-negative sum of PSD matrices is PSD, whence the result. \ifspringer\qed\fi
\end{proof}
A \textit{doubly non-negative} (DNN) matrix is a PSD matrix with non-negative entries. 
\begin{corollary}
  If $\Gamma\in\CONX{n}$, then it is DNN.
  \label{cor:dnn}
\end{corollary}
\begin{proof}
  By Lemma \ref{lem:necexact2} and the fact that every term of the RHS of Eq.~\eqref{eq:decomp} has non-negative entries.
\end{proof}

A \textit{completely positive} (CP) matrix $\Gamma$ is such that $\Gamma=A\transpose{A}$ where $A$ is $n\times r$ and $A$ is non-negative, i.e.
\begin{equation}
  \Gamma = \sum\limits_{k=1}^r A_{k}\transpose{A}_{k}.
  \label{eq:cp}
\end{equation}
The minimum $r$ for which $\Gamma$ can be represented as in Eq.~\eqref{eq:cp} is the {\it CP rank} of $\Gamma$. Any CP matrix is trivially DNN, but the converse does not hold in general \cite{berman}.
\begin{proposition}
If $\Gamma\in\CONX{n}$, then it is CP. \label{prop:cp}
\end{proposition}
\begin{proof}
  By definition, we have $\Gamma=\sum_{k\le r} p_K x_k\transpose{x}_k$ with $p\ge 0$ for some $r\le P_n$. Since $p\ge 0$, $\sqrt{p_k}$ is real for all $k$. Hence,
  \[\Gamma=\sum\limits_{k\le r}(\sqrt{p_K}x_k)\transpose{(\sqrt{p_k}x_k)}=\sum\limits_{k=1}^r A_{k}\transpose{A}_{k}, \]
  where $A_k=\sqrt{p_k}x_k$, as claimed.  \ifspringer\qed\fi
\end{proof}
Note that the $r$ in the proof of Prop.~\ref{prop:cp} does not correspond to the CP rank: while we show that membership of $\Gamma$ in $\CONX{n}$ also proves that $\Gamma$ is CP, there may exist other non-negative factorizations of $\Gamma$ in terms of an $n\times r'$ non-negative matrix $B$ with $r'<r$. The proof of Prop.~\ref{prop:cp} only shows that the CP rank cannot exceed $r$, but this is not a tight bound \cite[Thm.~3.5]{berman}.

Proving that a matrix is DNN can be done in polynomial time up to a given precision tolerance e.g.~by computing eigenvalues and inspecting the matrix entries. By contrast, proving that a matrix is CP is \textbf{NP}-hard by reduction from \textsc{max clique} via the Motzkin-Straus formulation \cite{motzkinstraus}, which can be transformed into a linear optimization problem over the CP matrix cone \cite{murty}. 

\section{New \textbf{NP}-hardness results}
\label{s:hardness}
In this section, we prove the \textbf{NP}-hardness of membership, rank, and relaxed rank in $\CONX{n}$, as well as of rank in $\COR{n}$, using Karp reductions. We also consider many corollaries of these four main theorems, and derive \textbf{NP}-hardness results for problems related to normalized polytopes and cut cones.

\subsection{Membership in $\CONX{n}$}
\label{s:feas}
We prove \textbf{NP}-hardness of the $\CONX{n}$ membership problem by reduction from the problem of establishing membership in the correlation polytope $\mathsf{COR}(n)$, known to be \textbf{NP}-hard by \cite{pitowsky} by reduction from \textsc{1-in-3sat}: given a \textsc{sat} formula $\psi$ in conjunctive normal form, each of whose clause consists of three literals, does there exist a satisfying solution for $\psi$ with the additional constraint that exactly one literal for each clause evaluates to TRUE?

Note that membership in $\COR{n}$ simply imposes a convex hull constraint $\sum_k p_k=1$ to the cone $\CONX{n}$. So, in our reduction, we model the satisfaction of this constraint as part of the cone. For this, we need one more dimension.

\begin{theorem}
  Membership in $\CONX{n}$ is \textbf{NP}-hard, by inclusion of the case where $\Gamma$ has its last row and column identical to the diagonal, and its last element equal to $1$.
  \label{thm:feas_nphard}
\end{theorem}
\begin{proof}
  All along the proof, we will employ symbols associated to ``z'' for the origin problem $\COR{n}$, and symbols associated to ``x'' for the target problem $\CONX{n}$.
  
For an instance $Z$ of $\mathsf{COR}(n)$ we will define a corresponding instance of $\CONX{n+1}$. More precisely, we let $\Gamma$ have its upper-left $n\times n$ block equal to $Z$, then $\Gamma_{i,n+1}=\Gamma_{n+1,i}=Z_{ii}$ for all $i\le n$, and finally $\Gamma_{n+1,n+1}=1$. This yields an affine map $\mathscr{L}:Z\mapsto\Gamma$ (also see Sect.~\ref{s:corrank}). Note that this definition makes $\Gamma$ have identical last row, last column, and diagonal.

We assume that $Z\in\mathsf{COR}(n)$: then, there exists a vector $\lambda\ge 0$ in $\mathbb{R}^{P_n}$ such that $Z=\sum_{k\le P_n} \lambda_kz_k\transpose{z}_k$ holds. For each vector $z_k\in\{0,1\}^n$, we define $x_k\in\{0,1\}^{n+1}$ as $x_k=(z_k,1)$ and $p_k=\lambda_k$ for all $k\le P_n$. Now we have
\begin{itemize}
\item for $1\le i\le j\le n$:
  \[\Gamma_{ij}=Z_{ij}=\sum_{k\le P_n} \lambda_k Z^k_{ij}=\sum_{k\le P_n} p_k Z^k_{ij}=\sum_{k\le P_n} p_k X^k_{ij}, \;\]
  where $X^k=x_k\transpose{x}_k$;
\item for $1\le i\le n$ (last column):
  \[\Gamma_{i,n+1}=Z_{ii}=\sum_{k\le P_n}\lambda_k Z^k_{ii}=\sum_{k\le P_n}\lambda_k z_{ki}^2,\]
  and note that $z_{ki}^2=z_{ki}$ (because $z\in\{0,1\}$), that $z_{ki}\times 1=x_{ki}x_{k,n+1}$ (by definition of $x_k$), and that $x_{ki}x_{k,n+1}=X^k_{i,n+1}$ (by definition of $X^k$), so $\Gamma_{i,n+1}=\sum_{k\le P_n} p_k X^k_{i,n+1}$ (by definition of $p_k$);
\item the last row of $\Gamma$ is identical to its last column by symmetry;
\item lastly, we check that this construction is consistent for $i=j=n+1$, i.e.
  \[1=\Gamma_{n+1,n+1}=\sum_{k\le P_{n+1}} p_kX^k_{n+1,n+1}=\sum_{k\le P_n} (p_k\times 1)=\sum_{k\le P_n} p_k=\sum_{k\le P_n} \lambda_k=1\]
  by definition of $\Gamma_{n+1,n+1}$, because $X^k_{n+1,n+1}=x_{k,n+1}^2=1$ whenever $x_k=(z_k,1)$ and $0$ otherwise, by definition of $p_k=\lambda_k$ for all $k\le P_n$, and because $Z$ is a YES instance of $\mathsf{COR}(n)$. 
\end{itemize} 
Therefore, $p\ge 0$ is such that $\Gamma=\sum_{k\le P_{n+1}} p_k X^k$, as claimed.

Next, we assume that $\Gamma$ (constructed from $Z$ as above) is an $(n+1)\times (n+1)$ YES instance of $\CONX{n}$. Thus, $\Gamma=\sum_{k\le P_{n+1}} p_k X^k$, with $X^k=x_k\transpose{x}_k$. For each $k$, let $x_k=(z_k,b)$ be a decomposition of $x_k$ into its first $n$ elements vector $z_k$, followed by the last element $b$. We know, by definition of $\Gamma$, that it has identical last row, last column, and diagonal. Hence, subtracting the last column from the diagonal, we have
\begin{equation}
  0=\sum_{k\le P_{n+1}} p_k(X^k_{ii} - X^k_{i,n+1})=\sum_{k\le P_{n+1}} p_k x_{ki}(1-x_{k,n+1})
  \label{eq:zeroeq}
\end{equation}
because $x_{ki}^2=x_{ki}\in\{0,1\}$. Since $p\ge 0$, by Eq.~\eqref{eq:zeroeq} we have $x_{ki}(1-x_{k,n+1})=0$ for all $k,i$, so if $x_{k,n+1}=b=0$, then $x_{ki}=0$ for all $i\le n$, whence $b=1$ for all non-zero boolean vectors $x_k$. This implies 
\begin{equation}
  1=\Gamma_{n+1,n+1}=\sum_{k\le P_{n+1}} p_k X^k_{n+1,n+1}=\sum_{k\le P_{n+1}} p_k x_{k,n+1}^2=\sum_{k\le P_n} p_k, 
  \label{eq:sumpk1}
\end{equation}
where the last sum consists of only $P_n$ terms because the vectors $x_k$ having zero $(n+1)$-st components are identically the zero vector, as argued above. Finally, the $n\times n$ top left block of $\Gamma$ being equal to $Z$ by construction, we have $Z=\sum_{k\le P_n} p_k X^k$ with $\sum_{k\le P_n} p_k=1$ by Eq.~\eqref{eq:sumpk1}. Therefore, $Y$ is a YES instance of $\mathsf{COR}(n)$, as claimed.
\end{proof}
The proof of Thm.~\ref{thm:feas_nphard} is such that, potentially, both source and target problems may have exponentially-sized YES certificates: this does not invalidate the reduction presented in the theorem, but it might be puzzling to some readers. By Carath\'eodory's theorem, for any long representation there must also exist a compact one. Moreover, both source and target problems are essentially LP feasibility problems: membership in $\mathsf{COR}(n)$ can be formulated as an LP feasibility with $N+1$ equality constraints, and membership in $\CONX{n}$ has $N$ equality constraints. Therefore, both problems admit basic feasible solutions with at most $N+1$ nonzeroes. These facts also prove that membership of $\CONX{n}$ is in \textbf{NP}. 

\subsection{Rank in $\CONX{n}$}
\label{s:rank}
In this section, we prove the \textbf{NP}-hardness of the rank problem associated to $\CONX{n}$.
\begin{quote}
  Given a positive integer $q$ and a symmetric $n\times n$ matrix $\Gamma$, find a vector $p\in\mathbb{R}^{P_n}_+$ such that
  \begin{eqnarray}
    \|p\|_0 &\le& q \label{eq:rk1} \\
    \Gamma &=& \sum_{k\le P_n} p_k x_k\transpose{x}_k. \label{eq:rk2}
\end{eqnarray}
\end{quote}

We reduce from \textsc{linear exact cover by 3-sets} (\textsc{linear}-X3C): given a set $U$ with $|U|=3q$ and a set $\mathcal{S}=\{S_1,\ldots,S_m\}$ of $m$ subsets of $U$ each having exactly three elements, and such that every unordered pair $\{e,f\}\subset U$ occurs in at most one of the subsets in $\mathcal{S}$, does $\mathcal{S}$ contain an exact cover (i.e.~a partition) of $U$ using exactly $q$ out of the $m$ subsets? If this problem seems contrived it is because its purpose is exactly to serve as reduction source for other, more interesting problems with challenging constraints. A non-peer-reviewed (but reasonably clear) proof of the \textbf{NP}-hardness of \textsc{linear}-X3C can be found in the first answer to \href{https://cstheory.stackexchange.com/questions/20386/}{question 20386} on \texttt{cstheory.stackexchange.com}. A peer-reviewed proof can be found in \cite[\S 2]{toulouse}, under the problem name PSTS-1-RES: given a partial Steiner triple system on $3q$ points (where $q=2\ell+1$ for some integer $\ell$), does it contain a parallel class, i.e.~a set of $q$ triplets such that each unordered pair of elements from $U$ occurs in at most one triplet of $\mathcal{S}$? The only limitation that $q$ should be odd can be removed by adding three new elements to $U$ and adding a triplet to $\mathcal{S}$ that contains exactly those new elements.

\begin{theorem}
  The rank problem for $\CONX{n}$ is \textbf{NP}-hard.
  \label{thm:rkhard}
\end{theorem}
\begin{proof}
  Given an instance of \textsc{linear}-X3C with $|U|=3q$ for some integer $q\ge 1$, we assume without loss of generality that $U=\{1,\ldots,n-1\}$, so $n=3q+1$. We construct $\Gamma$ as follows:
  \begin{itemize}
  \item for each $i<n$: $\Gamma_{in}=\Gamma_{ni} = 1$;
  \item for each $i<j<n$:
    $\Gamma_{ij} = \Gamma_{ji} = \left\{\begin{array}{ll} 1 & \mbox{if $\exists S\in\mathcal{S}$ with $\{i,j\}\in S$,} \\
      0 & \mbox{otherwise;}
    \end{array}\right.
    $
  \item for $i\in U$: $\Gamma_{ii} = 1$;
  \item $\Gamma_{nn}=q$.
  \end{itemize}
  
  First, we assume that the \textsc{linear}-X3C instance is YES. So, there exists an exact cover $\mathcal{T}\subset\mathcal{S}$ that partitions $U$: for every $i<n$ there is a unique $S\in\mathcal{T}$ such that $i\in S$. Since $|U|=3q$ and it has a partition into triplets, the number of triplets in the partition is $q$. So we can write $\mathcal{T}=\{T_1,\ldots,T_q\}$, where each $T_k\in\mathcal{T}$ is equal to some $S_j\in\mathcal{S}$. For each $k\le q$ we define $x_k\in\{0,1\}^n$ as follows:
  \[
  x_{kn}=1 \qquad \mbox{and} \qquad \forall i<n\quad x_{ki}=\left\{\begin{array}{ll} 1 & \mbox{if $i\in T_k$,} \\ 0 & \mbox{otherwise.}\end{array}\right.
  \]
  Moreover, we let $p_k=1$ for all $k\le q$. Let $\bar{\Gamma}=\sum_{k\le q} p_k x_k\transpose{x}_k$: we claim that $\bar{\Gamma}=\Gamma$. For each $i<n$ exactly one triplet $T_k$ in $\mathcal{T}$ contains $i$, hence exactly one boolean vector $x_k$ has $x_{ki}=x_{kn}=1$, which implies $\bar{\Gamma}_{in}=1=\Gamma_{in}$. The same consideration also implies that $\bar{\Gamma}_{ii}=\Gamma_{ii}$ for all $i<n$. For $i<j<n$ we have $\bar{\Gamma}_{ij}=1$ iff $i,j$ both belong to a given triplet $T_k$ (for some $k\le q$). By definition of \textsc{linear}-X3C, $T_k$ must be the only triplet to contain both $i,j$: hence $\bar{\Gamma}_{ij}=\Gamma_{ij}$. Therefore, $\Gamma$ satisfies Eq.~\eqref{eq:rk2} with exactly $q$ nonzero coefficients, whence the rank of $\Gamma$ is $\le q$.

  Next, we assume that $(q,\Gamma)$ is a YES instance of the rank problem for $\CONX{n}$, i.e.~$\Gamma$  has rank $r\le q$, whence $\Gamma=\sum_{k\le r} p_k x_k\transpose{x}_k$ with $p>0$. For each $k\le r$ we define the support sets $R_k=\{i<n\;|\;x_{ki}=1\}$ of each $x_k$. For each $i<n$ we have
  \begin{equation}
    1=\Gamma_{ii}=\sum_{k\le r} p_k x_{ki}^2=\sum_{k\le r\atop i\in R_k} p_k.
    \label{eq:ii}
  \end{equation}
  Moreover, since $\Gamma_{ni}=\Gamma_{in}=1$ for all $i<n$, we have
  \[1=\Gamma_{in}=\sum_{k\le r} p_k x_{kn}x_{ki}=\sum_{k\le r\atop i,n\in R_k} p_k.\]
  Subtracting one equation from the other, we infer that any $x_k$ such that $x_{ki}=1$ must also have $x_{kn}=1$, otherwise the two sums could not both  be equal to $1$. Hence, for any $k\le r$, if $R_k$ contains $i<n$ it also contains $n$. Now, for any $k\le r$ consider $i<j<n\in R_k$: the $k$-th term contributes $p_k>0$ to $\Gamma_{ij}$. Since $\Gamma\in\{0,1\}^{n\times n}$ and $p\ge 0$, $\Gamma_{ij}=1$. By construction of $\Gamma$, there must be $S_h\in\mathcal{S}$ such that $i,j\in S_h$. By definition of \textsc{linear}-X3C, all pairs in $R_k$ must be in a single triplet, so $R_k\subseteq S_h$, implying $|R_k|\le 3$.

  At this point, we claim that $r=q$, all $p_k=1$, and $|R_k|=3$. We sum over the last column of $\Gamma$ and obtain $\sum_{i<n}\Gamma_{in}=3q$, since $\Gamma_{in}=1$ for all $i<n$, and $|U|=|\{i\;|\;i<n\}|=3q$. We also have
  \[\sum_{i<n}\Gamma_{in}=\sum_{i<n}\sum_{k\le r}p_k x_{ki} x_{kn}=\sum_{k\le r}p_k|R_k| x_{kn}=\sum_{k\le r}p_k|R_k|\]
  by definition of $\Gamma$, by definition of $R_k$, and because $x_{kn}=1$ whenever $R_k\not=\varnothing$. Note that $|R_k|\le 3$ implies
  \begin{equation}
    3q = \sum_{k\le r} p_k|R_k|\le 3\sum_{k\le r}p_k.
    \label{eq:3q}
  \end{equation}
  Now, by $\Gamma_{nn}=q$ we have
  \begin{equation}
    q = \Gamma_{nn}=\sum_{k\le r} p_k x_{kn}^2 =\sum_{k\le r} p_k x_{kn}=\sum_{k\le r} p_k,
    \label{eq:qq}
  \end{equation}
  again by definition of $\Gamma$, by $x_{kn}\in\{0,1\}$, and because $x_{kn}=1$ whenever $p_k>0$. Combining Eq.~\eqref{eq:3q}-\eqref{eq:qq} we obtain
  \[3q\le 3\sum_{k\le r} p_k = 3q,\]
  whence equality must hold in Eq.~\eqref{eq:3q}. Therefore $|R_k|=3$ for all $k$ with $p_k>0$. To settle the claim, we remark the following facts: (i) $|R_k|=3$ for all $k\le r$, (ii) $\sum_{k\le r} p_k=q$, (iii) $r\le q$, and (iv) $p_k>0$ for all $k\le r$. These imply that $r=q$ and $p_k=1$ for all $k\le r$, as claimed.
  
Finally, by Eq.~\eqref{eq:ii} and the fact that $p_k=1$ for all $k\le r$, for each $i<n$ we have \[1=\sum_{k\le r\atop i\in R_k}1,\] so each $i<n$ belongs to exactly one $R_k$. Therefore, $R_1,\ldots R_r$ form a partition of $U$ into $q$ disjoint triplets: then $\mathcal{T}=\{R_1,\ldots R_r\}$ is an exact cover of $U$, which implies that the \textsc{linear}-X3C is a YES instance, as claimed.
\end{proof}

\subsection{Relaxed rank in $\CONX{n}$}
\label{s:rrank}
The relaxed rank problem with respect to $\CONX{n}$ is as follows. Given a symmetric $n\times n$ matrix $\Gamma\in\CONX{n}$ ($1$-bit promise), find $p\ge 0$ such that
\begin{equation}
  \left.\begin{array}{rcl}
  \sum\limits_{0\le k\le P_n} p_k X^k &=& \Gamma \\
  \sum\limits_{0\le k\le P_n} p_k &\le & \rho.
    \end{array}\right\} \label{eq:rrconx}
\end{equation}

\begin{theorem}
  The relaxed rank problem for $\CONX{n}$ is \textbf{NP}-hard.
  \label{thm:rrank}
\end{theorem}
\begin{proof}
We reduce from the $\COR{n}$ membership problem by exploiting its simplex constraint $\sum_k p_k=1$. For any $\rho>0$, we define the scaled polytope $\rho\COR{n}$ as
\begin{equation}
  \left.\begin{array}{rcl}
    \sum\limits_{k\ge 0} p_k X^k &=& \Gamma \\
    \sum\limits_{k\ge 0} p_k &= & \rho.
    \end{array}\right\} \label{eq:scaledcor}
\end{equation}
Since scaling by $\rho$ is linear and invertible, it is a linear isomorphism, and therefore, by \cite{gls}, it provides a Karp reduction from $\COR{n}$ membership to $\rho\COR{n}$ membership, making the latter \textbf{NP}-hard.

Next, we consider the membership problem:
\begin{equation}
  \left.\begin{array}{rcl}
    \sum\limits_{k\ge 0} p_k X^k &=& \Gamma \\
    \sum\limits_{k\ge 0} p_k &\le & \rho.
    \end{array}\right\} \label{eq:relaxedscaledcor}
\end{equation}
We prove that satisfying Eq.~\eqref{eq:relaxedscaledcor} is \textbf{NP}-hard, by reduction from $\sigma\COR{n}$ for some $0<\sigma\le\rho$. Let $\Gamma\in\sigma\COR{n}$. Then, $\Gamma$ trivially satisfies Eq.~\eqref{eq:relaxedscaledcor} because $\sum_{k\ge 0}p_k=\sigma\le\rho$. Conversely, take a $\Gamma$ that satisfies Eq.~\eqref{eq:relaxedscaledcor}, and assume (without loss of generality) that $p_0=0$. Therefore, there is some $\tau>0$ such that $\sum\limits_{k\ge 1}p_k=\tau\le\rho$. We set $p_0=\rho-\tau$. Now $\sum\limits_{k\ge 0} p_k=p_0+\sum\limits_{k\ge 1} p_k=\rho-\tau+\tau=\rho$. This shows that $\Gamma\in\rho\COR{n}$. Finally, we scale by $\sigma/\rho\le 1$ to show that $\Gamma\in\sigma\COR{n}$, which concludes the reduction. This shows that Eq.~\eqref{eq:relaxedscaledcor} is \textbf{NP}-hard. Lastly, we note that Eq.~\eqref{eq:relaxedscaledcor} is also the definition of the relaxed rank problem for $\CONX{n}$, which is therefore \textbf{NP}-hard. 
\end{proof}
A non-geometric, longer, and more informative proof by reduction from the {\sc Fractional Clique Cover} problem can be found in Appendix \ref{a:rrank}.

\subsection{Rank in $\COR{n}$}
\label{s:corrank}
In order to prove that the rank decision problem for $\COR{n}$ is \textbf{NP}-hard, we use reduction techniques from Sect.~\ref{s:feas}-\ref{s:rank}. More precisely, we reduce from the rank problem in $\CONX{n}$ and use the same lifting as in the reduction from $\COR{n}$ membership to $\CONX{n+1}$ membership. For an $n\times n$ matrix $Z$, we employ the affine map $\mathscr{L}$ defined in the proof of Thm.~\ref{thm:feas_nphard}:
\[\mathscr{L}(Z) = \left(\begin{array}{cc} Z & \diag{Z} \\ \transpose{\diag{Z}} & 1\end{array}\right).\]
  \begin{lemma}
    \label{lem:ZL}
For any $n\times n$ symmetric matrix $Z$ and any integer $\rho\ge 1$, we have
\begin{enumerate}[(i)]
\item if $Z\in\COR{n}$, then $\mathscr{L}(Z)\in\CONX{n+1}$;
\item  $Z\in\COR{n}$ and the $\COR{n}$-rank of $Z$ is $\le\rho$ $\Longleftrightarrow$ $\mathscr{L}(Z)\in\CONX{n+1}$ and the $\CONX{n+1}$-rank of $\mathscr{L}(Z)$ is $\le\rho$.
\end{enumerate}
\end{lemma}
\begin{proof}
  ($\Rightarrow$) Assume $Z=\sum_{k=0}^r\lambda_k z_k\transpose{z}_k$, $\lambda\ge 0$, $\sum_{k=0}^r\lambda_k=1$, $r\le\rho$. Let $x_k=(z_k,1)\in\{0,1\}^{n+1}$. Then,
  \begin{eqnarray*}
    x_k\transpose{x}_k &=& \left(\begin{array}{cc} z_k\transpose{z}_k & z_k \\ \transpose{z}_k & 1 \end{array}\right) \\
    \Rightarrow \sum_{k=0}^r \lambda_k x_k\transpose{x}_k &=& \left(\begin{array}{cc} \sum_k\lambda_k z_k\transpose{z}_k & \sum_k \lambda_k z_k \\
      \sum_k\lambda_k \transpose{z}_k & \sum_k\lambda_k\end{array}\right) =
      \left(\begin{array}{cc} Z & \diag{Z} \\ \transpose{\diag{Z}} & 1\end{array}\right)=\mathscr{L}(Z)=\Gamma
  \end{eqnarray*}
  So $\Gamma\in\CONX{n+1}$ and uses $r\le\rho$ generators.

  ($\Leftarrow$) Assume $\Gamma=\sum_{k=0}^r p_k x_k\transpose{x}_k$, $p\ge 0$, $x_k\in\{0,1\}^{n+1}$ for all $k\le r$, and $r\le\rho$. The $(n+1,n+1)$-th entry of $\Gamma$ is $1$, and in particular $1=\Gamma_{n+1,n+1}=\sum_{k=0}^r p_k x_{k,n+1}^2 = \sum_{k=0}^r p_k x_{k,n+1}$ because booleans are square-invariant. Since $x_{k,n+1}\in\{0,1\}$ and $p\ge 0$, we have $x_{k,n+1}=1$ for every $k$ with $p_k>0$, and $\sum_{k=0}^r p_k=1$. Again, since $x_k=(z_k,1)$, the upper-left block of $\mathscr{L}(Z)$ is $Z=\sum_{k=0}^r p_k x_k\transpose{x}_k$ with $\sum_{k=0}^r p_k=1$, whence $Z\in\COR{n}$ and the rank of $Z$ is $r\le\rho$, as claimed. 
\end{proof}    

\begin{theorem}
  The rank problem for $\COR{n}$ is \textbf{NP}-hard.\label{thm:corrank}
\end{theorem}
\begin{proof}
  We reduce from the rank problem in $\CONX{n+1}$, which is \textbf{NP}-hard by Thm.~\ref{thm:rkhard}. By Lemma \ref{lem:ZL}, the map $\mathscr{L}:Z\mapsto \mathscr{L}(Z)$ is affine, can be carried out in polynomial time, and maps rational instances to rational instances. Therefore, by \cite{gls}, it provides a valid polynomial reduction from $\CONX{n+1}$ to $\COR{n}$, as claimed. We note that, by Lemma \ref{lem:ZL}, $Z$ has $\COR{n}$-rank at most $\rho$ iff $\mathscr{L}(Z)$ has $\CONX{n+1}$-rank at most $\rho$.
\end{proof}

\subsection{Hardness corollaries}
\label{s:corollaries}
In this section we list several corollaries of the theorems in Sect.~\ref{s:hardness}.

\subsubsection{Membership problems in normalized polytopes}
\label{s:normalized}
We first introduce the membership problems for the normalized polytopes for correlation and cut cones, which will give us a useful affine isomorphism to be used later. Note that $\COR{n}$ does not satisfy the definition of ``convex hull of rank-one boolean matrices'' syntactically, since the zero matrix has rank zero, not one. By contrast, $\nCOR{n}$ does. We can reduce membership in $\COR{n}$ to membership in $\nCOR{n+1}$ by the mapping $\mathscr{L}'$ (reflection of $\mathscr{L}$ along the skew diagonal, see the proof of Thm.~\ref{thm:feas_nphard}) that sends $\Gamma\in\COR{n}$ to the bordered matrix
\begin{equation}
\Gamma' = \left(\begin{array}{cc} 1 & \transpose{\diag{\Gamma}} \\ \diag{\Gamma} & \Gamma\end{array}\right) \in\nCOR{n+1},
\label{eq:ncor}
\end{equation}
where the presence of the $1$ in the upper-left corner ensures that the affine map never sends any $\Gamma\in\COR{n}$ to the zero matrix. Thus, $\mathscr{L}'$ is a polytime-computable and rational affine isomorphism between a $\COR{n}$ and a face of $\nCOR{n+1}$. Therefore, by \cite{pitowsky} and \cite{gls}, we obtain a Karp reduction from $\COR{n}$ to $\nCOR{n+1}$ (by inclusion of the relevant face), showing that membership in $\nCOR{n}$ is also \textbf{NP}-hard. 

Starting from $\CUT{n}$, we define a \textit{normalized cut polytope}, denoted $\nCUT{n}$, which removes the extreme point $Y^0=y_0\transpose{y}_0=(-\mathbf{1})\transpose{(-\mathbf{1})}=\mathbf{1}\transpose{\mathbf{1}}=Y^{P_n}$. Similarly to the case of $\COR{n}$, we reduce $\CUT{n}$ to $\nCUT{n+1}$ by means the affine isomorphism in Eq.~\eqref{eq:ncor}, which, again by \cite{gls}, provides a reduction from $\CUT{n}$ to $\nCUT{n+1}$, which is therefore \textbf{NP}-hard. We remark that, at this point of our paper, the \textbf{NP}-hardness of $\CUT{n}$ is based on a Cook-Turing reduction (see Sect.~\ref{s:membership}). We correct the situation in Sect.~\ref{s:karpred} below.

\subsubsection{Karp reduction for membership in $\CUT{n}$}
\label{s:karpred}
Since we promised Karp reductions, but some of the results we cited from the literature are based on Cook-Turing reductions (Sect.~\ref{s:membership}), we provide a Karp reduction for proving \textbf{NP}-hardness of membership in $\CUT{n}$. We reduce from $\COR{n}$, extending the cardinality-wise correspondence between $\COR{n}$ and $\CUT{n+1}$ given in Sect.~\ref{s:cutpolycone} to an affine isomorphism. Consider the affine transformation from $x_k\in\{0,1\}$ to $y_k\in\{-1,1\}$ vectors again: for any $0\le k\le P_n$ we have $y_k=2x_k-\mathbf{1}$. From this, we obtain
\begin{equation}
  y_k\transpose{y}_k=4x_k\transpose{x}_k-2(x_k\transpose{\mathbf{1}}+\mathbf{1}\transpose{x}_k)+\mathbf{1}\transpose{\mathbf{1}}.
  \label{eq:y(x)}
\end{equation}
And from the inverse mapping $x_k=(y_k+\mathbf{1})/2$ we have
\begin{equation}
  x_k\transpose{x}_k=(y_k\transpose{y}_k+y_k\transpose{\mathbf{1}}+\mathbf{1}\transpose{y}_k + \mathbf{1}\transpose{\mathbf{1}})/4.
  \label{eq:x(y)}
\end{equation}
Using Eq.~\eqref{eq:y(x)}-\eqref{eq:x(y)} and $X^k=x_k\transpose{x}_k$, $Y^k=y_k\transpose{y}_k$, we can define an affine isomorphism $\mathscr{C}$ from $\COR{n}$ and $\CUT{n+1}$ as follows:
\begin{itemize}
\item for $X^k\in X$ define $Y^k_{00}=1$, $Y^k_{0i}=2X^k_{ii}-1$ for $1\le i\le n$, and $Y^k_{ij}=4X^k_{ij}-2X^k_{ii}-2X^k_{jj}+1$ for $1\le i<j\le n$;
\item the inverse map is given by $X^k_{ij}=(1+Y^k_{0i}+Y^k_{0j} + Y^k_{ij})/4$ for all $1\le i\le j\le n$,
\end{itemize}
where we note that $Y$ is a symmetric matrix (which defines the remaining entries).

It is easy to show that this map is affine \cite{desimone90}: it suffices to verify that it holds for a generic matrix $\Gamma$ within $\COR{n}$ by simply applying the transformations to each term of the decomposition of $\Gamma$ in terms of the correlation polytope. Note, moreover, that this affine map can be applied in polynomial time, and maps rational instances of $\COR{n}$ to rational instances of $\COR{n+1}$. By \cite{gls}, a polynomial-time computable affine isomorphism with polynomially bounded encoding size gives a valid polynomial-time reduction between the corresponding decision problems. We therefore obtain a Karp reduction establishing the \textbf{NP}-hardness of the membership problem for $\CUT{n}$.

\subsubsection{Rank problem for the normalized correlation polytope}
The \textbf{NP}-hardness of membership in the normalized correlation polytope $\nCOR{n}$ was argued in Sect.~\ref{s:normalized}. The rank problem in $\COR{n}$ is \textbf{NP}-hard by Thm.~\ref{thm:corrank}, so we reduce from $\COR{n}$-rank to $\nCOR{n}$-rank in order to prove \textbf{NP}-hardness of the latter. To this purpose, we consider the same mapping $\mathscr{L}'$ defined in Sect.~\ref{s:normalized}, which maps every $Z\in\COR{n}$ to $\mathscr{L}'(Z)\in\nCOR{n+1}$, a bordered matrix with $Z$ in the upper left corner, $\diag{Z}$ in the borders, and the scalar $1$ in the lower-right entry. Since the $1$ entry is fixed, $\mathscr{L}'$ never maps to the zero matrix: hence $\mathscr{L}'$ maps into $\nCOR{n+1}$. Moreover, by Lemma \ref{lem:ZL} (with all sum quantifiers starting from $k=1$, which has no impact on the value of the sums since $X^0=0$), $\mathscr{L}'$ preserves the rank exactly. This provides a valid polynomial reduction.

\subsubsection{Rank problems for the cut and normalized cut polytopes}
The \textbf{NP}-hardness of the rank problems for $\CUT{n}$ and $\nCUT{n}$ follows from those of $\COR{n}$ and $\nCOR{n}$ and the affine mapping $\mathscr{C}$ given in Eq.~\eqref{eq:y(x)}-\eqref{eq:x(y)}, since affine isomorphisms such as $\mathscr{C}$ preserve rank exactly. This provides valid polynomial reductions.

\subsubsection{Membership, rank, relaxed rank for the cut cone}
The \textbf{NP}-hardness of membership, rank, and relaxed rank in the cut cone follows from those in the correlation cone $\CONX{n}$ (Sections \ref{s:feas}, \ref{s:rank}, \ref{s:rrank}) and the affine mapping $\mathscr{C}$ given in Eq.~\eqref{eq:y(x)}-\eqref{eq:x(y)}, since affine isomorphisms preserve rank. Moreover, because an isomorphism between cones must be linear (no translation terms), the mapping $\mathscr{C}$ applied to cones also preserves the relaxed rank $\|p\|_1$ that only depends on the conic hull and its generators, which are both preserved under the mapping. This provides valid polynomial reductions.

\section{Related results}
\label{s:related}
In this section, we look at easy polynomial cases of the problems that we proved to be \textbf{NP}-hard. Finally, we also show that the exponential extension complexity of the $\COR{n}$ membership problem implies the same for the $\CONX{n}$ membership problem. The fact that $\COR{n}$ has exponential extension complexity is not overly surprising given that the problem is \textbf{NP}-hard, but \cite{kaibelWeltge} discriminates between super-polynomial and exponential, and unconditionally with respect to the conjecture $\mathbf{P}\not=\mathbf{NP}$.

\subsection{Simple polynomial cases}
\label{s:polycases}
Let $G=(V,E)$ be the support graph of the $n\times n$ symmetric matrix $\Gamma$.

\subsubsection{Forests}
If $G$ is a forest (e.g.~tree, matching, star), then $\Gamma$ can be covered by boolean vector support sets having cardinality limited to $1,2$, as long as
\begin{equation}
\forall i\le n\quad s_i=\Gamma_{ii} - \sum_{j\le n\atop \{i,j\}\in E}\Gamma_{ij}\ge 0. \label{eq:poly1}
\end{equation}
If Eq.~\eqref{eq:poly1} holds, then we have
\[\Gamma=\sum_{\{i,j\}\in E} \Gamma_{ij} X^{\{i,j\}} + s_i X^{i},\]
where $X^{\{i,j\}}$ are the rank-one boolean matrices obtained by support vectors with two nonzero entries corresponding to indices $i,j$, and $X^{i}$ are those with one nonzero entry corresponding to index $i$. 

\subsubsection{Bounded treewidth}
If $G$ has bounded treewidth, we can solve the three problems (membership, rank, relaxed rank) in polytime by Dynamic Programming (DP) on the tree decomposition. Suppose that the treewidth of $G$ is equal to $t$: then every clique in $G$ has size bounded by $t+1$, and every clique is contained in some node (called a ``bag'', and representing a set of vertices of the original graph) of any tree decomposition of $G$ having width $t$. These are the premises to a general DP algorithm based on bounded treewidth graphs \cite{bodlaender2}.

Moreover, it can be shown that the exponentially long sums in the equality constraints in Eq.~\eqref{eq:decomp} whenever $r=P_n$ become polynomial when $G$ has bounded treewidth. Enumerate all cliques of $G$ by listing all non-empty subsets $C$ in each bag of a nice tree decomposition of $G$ (i.e.~a rooted tree decomposition where every node has at most two children and every node belongs to one of four types \cite{kloks}), then verify if they are cliques. Since the treewidth $t$ is a constant, the set $\K$ of these cliques has cardinality $O(2^tn)=O(n)$. Now we introduce the $O(n)$ decision variables $p_C\ge 0$ for each clique $C$, and the $O(n^2)$ equality constraints:
\begin{eqnarray}
\Gamma_{ij} &=& \sum_{C \in \K\atop i,j\in C} p_C \label{eq:clique1} \\
\Gamma_{ii} &=& \sum_{C \in \K\atop i\in C} P_C. \label{eq:clique2}
\end{eqnarray}
Eq.~\eqref{eq:clique1}-\eqref{eq:clique2} can replace Eq.~\eqref{eq:decomp} in the three exponential LP formulations for membership and relaxed rank, yielding polynomially-sized LP formulations for membership and relaxed rank. Applied to the exact rank problem, we obtain a polynomially-sized MILP formulation with $O(n)$ variables, which can be solved by the methods given in \cite{chan}. The exact rank problem can also be solved by DP \cite{bodlaender2} because $\|p\|_0$ is additive over the covering cliques from the bags of the tree decomposition.

\subsubsection{Chordality}
If $G$ is chordal, then every clique is contained in a unique maximal clique, the maximal cliques form a clique tree, and there are at most $n$ maximal cliques \cite{chordalGraph}, which can therefore be listed in polynomial time. Thus, both of the exponential LPs for membership and relaxed rank become polynomially-sized: both can then be solved in weakly polynomial time using e.g.~the interior point method for LP.

\subsubsection{Perfect graphs}
If $G$ is a perfect graph, the maximum weight clique is polynomial-time solvable. The primal LP formulation of the relaxed rank problem is:
\[\min \{\sum_{C\in\K} p_C \;|\; \forall i\le j\le n\; \Gamma_{ij}=\sum_{C\in\K\atop i,j\in C} p_C\land p\ge 0\},\]
where $\K$ is the set of all cliques in $G$. Its dual is
\[\max \{\langle Y,\Gamma\rangle \;|\;\forall C\in\K \;\sum_{i,j\in C} Y_{ij}\le 1\},\]
which has an exponential number of constraints. These can be handled in polynomial time by separation, iteratively solving the separation subproblem ``given $Y$, find a clique $C\in\K$ maximizing $\sum_{ij\in C} Y_{ij}$'' \cite{perfectGraph}. 

\subsection{Extension complexity}
Let $\mathsf{xc}(P)$ be the extension complexity of a linear programming formulation of some given problem $P$. Consider the affine slice 
\[ \Sigma = \{ \Gamma\in\CONX{n+1} \;|\;\Gamma_{n+1,n+1}=1 \} \]
of $\CONX{n+1}$, which is affinely isomorphic to $\COR{n}$ by the mapping $\mathscr{L}$ introduced in the proof of Thm.~\ref{thm:feas_nphard}. We have:
\begin{equation*}
\mathsf{xc}(\CONX{n+1})\ge\mathsf{xc}(\Sigma)=\mathsf{xc}(\COR{n})=2^{\Omega(n)},
\end{equation*}
where the last equation follows by \cite{kaibelWeltge}. Since Eq.~\eqref{eq:decomp} has extension complexity $2^n$, we conclude that
\[\exists c>0 \quad 2^{cn} \le \mathsf{xc}(\CONX{n})\le 2^n.\]

\section{Conclusion}
\label{s:concl}
We proved \textbf{NP}-hardness of three fundamental problems in polyhedral theory (membership, rank, and a relaxed rank derived from relaxing the zero-norm to the rank-norm) when applied to the correlation cone and polytope, the boolean quadric cone and polytope, and the cut cone and polytope, as well as to the normalized versions of the mentioned polytopes. All of these results stem from five reductions on the correlation cone and polytope, one of which (membership in the correlation polytope) was already known \cite{pitowsky}, while we believe that the other four are new.

\section*{Acknowledgments}
LL was partially supported by Gruppo Nazionale per l'Analisi Matematica la Probabilità e le loro Applicazioni – INDAM, project CUP: E53C24001950001 and by CNR Short Term Mobility Program.

\bibliographystyle{plain}
\bibliography{sb1d}

\ifelsevier
  \appendix
\else
  \ifspringer
    \appendix
  \else
  \appendix
    \titleformat{\section}
    {\normalfont\Large\bfseries}
    {Appendix \thesection}{1em}{}
  \fi
\fi

\section{Alternative reduction for the $\CONX{n}$ relaxed rank}
\label{a:rrank}
Let $\mathsf{K}$ be the complete undirected graph on the vertex set $V(\mathsf{K})=\{1,\ldots,n\}$, with loops. We denote non-loop edges by $\{i,j\}\in E(\mathsf{K})$ and loops by the singletons $\{i\}$. For $y\in\{0,1\}^n$ we let $Y=y\transpose{y}$ and $K(y)$ be the complete graph on the support set of the vector $y$, so that $Y$ is the adjacency matrix of $K(y)$. 

We re-cast the relaxed rank problem with respect to $\CONX{n}$ as follows. Given a scalar $\rho>0$ and a matrix $\Gamma$, find a collection $\mathcal{V}$ of vectors in $\{0,1\}^n$ and a vector $p\ge 0$ such that
\begin{eqnarray}
\sum\limits_{0\le k\le P_n} p_k &\le& \rho \label{eq:rr1} \\
\forall\{i,j\}\in E(\mathsf{K})\quad \sum\limits_{x_k\in\mathcal{V} \atop \{i,j\}\in K(x_k)} p_k &=&\Gamma_{ij}. \label{eq:rr2}
\end{eqnarray}

We prove the \textbf{NP}-hardness of the relaxed rank problem related to $\CONX{n}$ by reduction from the {\sc Fractional Clique Cover} problem, which is also known as {\sc Fractional Vertex Coloring} on the complemented graph \cite{lundyann}. 
\begin{quote}
  {\sc Fractional Clique Cover} (FCC). Given an undirected graph $G=(V,E)$ and a value $t>0$, determine whether there exists a set $\mathcal{K}$ of cliques of $G$, and a vector $w\in\mathbb{R}^{|\mathcal{K}|}_+$, such that
 \begin{eqnarray}
 \sum\limits_{C\in\mathcal{K}}w_{C} &\le& t \label{eq:fcc1} \\
  \forall v\in V\quad \sum\limits_{C\in\mathcal{K}\atop v\in C} w_C &=& 1. \label{eq:fcc}
 \end{eqnarray}
\end{quote}
The reduction maps Eq.~\eqref{eq:fcc1} to Eq.~\eqref{eq:rr1} and Eq.~\eqref{eq:fcc} to Eq.~\eqref{eq:rr2}.

\begin{theorem}
  The relaxed rank problem for $\CONX{n}$ is \textbf{NP}-hard by inclusion of the case where the off-diagonal entries of $\Gamma$ are in $\{0,\alpha\}$ for some $\alpha>0$.
  \label{thm:rrank2}
\end{theorem}
\begin{proof}
For a given scalar $t>0$ and a graph $G=(V,E)$, let $(t,G)$ be an instance of the FCC. We construct an instance of the relaxed rank problem with $n=|V|+1$, so we assume that $V=\{1,\ldots,n-1\}$: for $i,j\in V, i\neq j$, let $\Gamma_{ij}=1/n^2$ for $\{i,j\}\in E$ and $\Gamma_{ij}=0$ for $\{i,j\}\notin E$. Moreover, for $i\in V$, let $\Gamma_{ii}=1/n$. Then let $\Gamma_{in}=1/n^2$ for $i\in V$ and $\Gamma_{nn}=t/n^2$. Finally, we let \leo{$\rho=\frac{2t}{n^2} + \frac{3n-1}{2n}=\frac{3n^2-n+4t}{2n^2}$}. We prove that the instance \leo{$(\rho,\Gamma)$} is YES iff $(t,G)$ is a YES instance of the FCC. Note that the size of the relaxed rank instance is polynomially bounded in the size of the FCC instance.

We assume that the solution of the YES instance $(t,G)$ of the FCC is a collection $\C$ of (say) $m$ cliques of $G$ with weights $w_C>0$ for $C\in\C$, satisfying Eq.~\eqref{eq:fcc1}-\eqref{eq:fcc}. We note that $m$ is polynomially bounded in the size of the FCC instance, since NP certificates are polytime verifiable; therefore, $m$ is also polynomially bounded in the size of the relaxed rank instance.

We construct a solution of the corresponding relaxed rank instance $(\leo{\rho},\Gamma)$ consisting of a set $\K$ of cliques of the complete graph $\mathsf{K}$, and a vector $p$ indexed by cliques in $\K$. Initially, we set $\K=\varnothing$. We first deal with the last column of $\Gamma$: let $\K=\{C\cup\{n\}\;|\;C\in\C\}$. For a clique $K=C\cup\{n\}\in\K$, we define $p_{K}=w_{C}/n^2$. Condition \eqref{eq:fcc} guarantees that
\[\forall i\le n\quad \sum\limits_{K\in\K\atop \{i,n\}\subseteq K} p_K = \sum\limits_{C\in\C\atop i\in C}\frac{w_C}{n^2} = \frac{1}{n^2} = \Gamma_{in}.\]
Moreover, the contribution of these cliques to $\Gamma_{ij}$ for $\{i,j\}\not\in E$ is zero, as desired, i.e.
\[\forall \{i,j\}\not\in E\quad \sum\limits_{K\in\K\atop\{i,j\}\subseteq K} p_K = \sum\limits_{C\in\C\atop\{i,j\}\subseteq C} \frac{w_C}{n^2} = 0 = \Gamma_{ij},\]
since all sets in $\C$ are cliques of $G$, and $\{i,j\}\not\in E$. Next, the condition $\sum_{C\in\C} w_C \leq t$ guarantees that
\[\sum\limits_{K\in\K\atop n\in K} p_K = \sum\limits_{C\in\C} \frac{w_C}{n^2} \leq \frac{t}{n^2} = \Gamma_{nn}.\]
If the inequality is strict, we extend $\K$ with the $1$-clique $\{n\}$ having weight \leo{$p_{\{n\}}=\frac{t}{n^2}-\sum_{C\in\C} \frac{w_C}{n^2}$}, yielding
\[\sum\limits_{K\in\K\atop n\in K} p_K = \Gamma_{nn}.\]
\leo{We note that the clique $\{n\}$ is not yet part of $\K$, so we add it to $\K$.}

We now consider Eq.~\eqref{eq:decomp} for $\{i,j\}\in E$ (for both $i=j$ and $i\not=j$). We want to achieve
\[\forall \{i,j\}\in E\quad \sum_{K\in\K\atop\{i,j\}\subseteq K} p_K = \Gamma_{ij}.\]
With the current definition of $\K$, we have
\[\sum\limits_{K\in\K\atop\{i,j\}\subseteq K} p_K = \sum\limits_{C\in\C\atop\{i,j\}\subseteq C} \frac{w_C}{n^2} \leq \sum\limits_{C\in\C\atop i\in C} \frac{w_C}{n^2} = \frac{1}{n^2}.\]
In order to reach the required $\Gamma_{ij} = 1/n^2$, it is sufficient to \leo{consider the $2$-clique $\{i,j\}$ with weight}
\[p_{\{i,j\}}=\frac{1}{n^2}-\sum\limits_{K\in\K\atop\{i,j\}\subseteq K} p_K.\]
\leo{We note that the $2$-clique $\{i,j\}$ is not yet in $\K$, so we add it to $\K$.} This settles all $\Gamma_{ij}$ for $\{i,j\}\in E$.

For $i=j$, note that
\[ \forall i\in V\quad \sum\limits_{K\in\K\atop i\in K} p_K = \frac{1}{n^2}\]
after the initialization of \leo{the weights $p_K$ for $K\in \K$}. In the (possible) successive $2$-clique extensions, this sum is increased by at most $1/n^2$ for each edge incident to $i$ \leo{(there are at most $n-1$ such edges)}. This means that, after the extension, we have
\leo{\[\sum\limits_{K\in\K\atop i\in K} p_K \leq \frac{1}{n^2} + \frac{n-1}{n^2}=\frac{1}{n}.\]}
Hence, the required weight $\Gamma_{ii}$ is reached by further extending $\K$ (if needed) with the $1$-clique $\{i\}$ having weight
\[p_{\{i\}}=\frac{1}{n}-\sum\limits_{K\in\K\atop i\in K} p_K.\]
None of these $1$- and $2$-clique extensions affect the correct entries already achieved for $\Gamma$.

The construction of $\K$ ensures that its size, say $r$, is polynomially bounded in $m$, which is polynomially bounded in the size of the relaxed rank instance.

Finally, we look at Eq.~\eqref{eq:rr1}. We note that $\mathcal{W}=\sum_{K\in\K} p_K$ is initially at most $t/n^2$. The possible extensions of $\K$ increase $\mathcal{W}$ \leo{by at most: $\frac{t}{n^2}$ for $p_{\{n\}}$, $1/n^2$ for each incident edge to $v$, and \leo{$1/n$} for each $v\in V$}. So the total contribution $\mathcal{W}'$ from the extra cliques is bounded above by
\leo{\[\mathcal{W}'=\frac{t}{n^2} + \frac{n(n-1)}{2n^2} + n\frac{1}{n} = \frac{t}{n^2}+\frac{3n-1}{2n}, \]
and therefore $\mathcal{W}\le \frac{t}{n^2} + \mathcal{W}' = \frac{2t}{n^2} + \frac{3n-1}{2n}=\rho$}, as claimed. 
Thus, the reduction maps a feasible instance of the FCC to a feasible relaxed rank instance.

Conversely, consider a feasible solution of the relaxed rank instance, defined by the collection of cliques $\K$ of the complete graph $\mathsf{K}$ on $n$ vertices, with each clique having weight $p_K$ for $K\in\K$. Note that, for each $K\in\K$, $K\setminus\{n\}$ is a clique of $G$: this occurs because, for all $\{i,j\}\notin E$, $\Gamma_{ij}=0$ guarantees that $\{i,j\}\not\subseteq K$. Let $K_1,\ldots,K_m$ be the cliques in $\K$ such that $n\in K$ and $K\cap V\not=\varnothing$. We construct a feasible solution of the FCC instance $(t,G)$: we let $C=\{C_1,\ldots,C_m\}$, where $C_i=K_i\setminus\{n\}$ with weight $w_{C_i}=n^2 p_{K_i}$ for $i\in\{1,\ldots,m\}$. The condition
\[\Gamma_{nn} = \frac{t}{n^2} = \sum\limits_{K\in\K\atop n\in K} p_K \geq \sum\limits_{K\in\{K_1,\ldots,K_m\}} p_K\]
implies
\[\sum\limits_{C\in\C} w_C \leq t,\]
so the solution has value at most $t$. Finally, the condition
\[\forall i\in V\quad \Gamma_{in} = 1/n^2 = \sum\limits_{K\in\K\atop\{i,n\}\in K} p_K = \sum\limits_{K\in\{K_1,\ldots,K_m\}\atop i\in K} p_K\]
implies \eqref{eq:fcc}, so the solution is feasible. \ifspringer\qed\fi
\end{proof}
We observe that the relaxed rank problem for $\CONX{n}$ is in \textbf{NP}, given that YES instances can be certified by a number of strictly positive $p_k$ that is at most $N=n(n+1)/2$.

\end{document}